\documentclass[a4paper,10pt,reqno]{amsart}
\usepackage[utf8]{inputenc}
\usepackage[pdftex]{graphicx}
\usepackage{a4wide}
\usepackage[english]{babel}
\usepackage[T1]{fontenc}
\usepackage{amsmath,amssymb,amsthm,stmaryrd}
\usepackage{mathrsfs}
\usepackage{esint}
\usepackage{mathtools}
\usepackage{comment}
\usepackage{extarrows}

\newcommand{\R}{\mathbb{R}}

\newcommand{\E}{\mathbb{E}}
\newcommand{\Q}{\mathscr{Q}}
\newcommand{\N}{\mathbb{N}}
\newcommand{\Z}{\mathbb{Z}}
\newcommand{\F}{\mathscr{F}}

\newcommand{\Pro}{\mathbb{P}}

\newcommand{\A}{\mathscr{A}}

\newcommand{\D}{\mathscr{D}}

\newcommand{\epsb}{\partial_\varepsilon}
\newcommand{\G}{\mathscr{G}}
\newcommand{\lev}{\text{lev}}

\newtheorem{theorem}[equation]{Theorem}
\newtheorem*{theorem*}{Theorem}
\newtheorem{lemma}[equation]{Lemma}

\newtheorem{proposition}[equation]{Proposition}

\theoremstyle{definition}
\newtheorem{defin}[equation]{Definition}
\newtheorem*{defin*}{Definition}

\newtheorem{remark}[equation]{Remark}
\newtheorem*{remark*}{Remark}

\numberwithin{equation}{section}

\title{Adjacent dyadic systems and the $L^p$-boundedness of shift operators in metric spaces revisited}

\author{Olli Tapiola}

\address{Olli Tapiola, Department of Mathematics and Statistics, P.O.B. 68 (Gustaf H\"allstr\"omin katu 2b), FI-00014 University of Helsinki, Finland}
\email{olli.tapiola@helsinki.fi}

\date{\today}

\begin{document}

\begin{abstract}
  With the help of recent adjacent dyadic constructions by Hytönen and the author, we give an alternative proof of results of Lechner, Müller and Passenbrunner
  about the $L^p$-boundedness of shift operators acting on functions $f \in L^p(X;E)$ where $1 < p < \infty$, $X$ is a metric space and $E$ is a UMD space.
\end{abstract}

\maketitle

\section{Introduction}

During the last two decades, the highly influential $T(1)$ theorem of G. David and J.-L. Journé \cite{davidjourne} has been generalized to various 
settings by different authors (e.g. \cite{frazieretall, hanhofmann}). One of these generalizations was due to T. Figiel (\cite{figiel_singular, figiel_haar}, different 
proof by T. Hytönen and L. Weis \cite{hytonenweis}) who proved the theorem for UMD-valued functions $f \in L^p(\R^d; E)$ and scalar-valued kernels using a clever 
observation that any Caldéron-Zygmund operator on $\R^d$ can be decomposed into sums and products of Haar shifts (or rearragements), Haar multipliers and paraproducts. Not
long ago, P.F.X Müller and M. Passenbrunner \cite{mullerpassenbrunner} extended this technique from the Euclidean setting to metric spaces
to prove the $T(1)$ theorem for UMD-valued functions $f \in L^p (X;E)$, where $X$ is a normal space of homogeneous type (see Theorems 2 and 3 in \cite{maciassegovia}). One 
of the key elements of their (and Figiel's) proof - the $L^p$-boundedness of the shift operators - was revisited and simplified by R. Lechner and Passenbrunner in their 
recent paper \cite{lechnerpassenbrunner} by proving the result in a more general form with different techniques.

Roughly speaking, a shift operator permutates the generating Haar functions in such a way that if $h_Q \mapsto h_P$, then the dyadic cubes $P$ and $Q$ 
are not too far away from each other and they belong to the same generation of the given dyadic system. On the real line, this can be expressed in a very simple
form: for every $m \in \Z$, the shift operator $T_m$ is the linear extension of the map $h_I \mapsto h_{I + m|I|}$. In \cite[Theorem 1]{figiel_haar}, Figiel showed 
that for UMD-valued functions $f \colon [0,1] \to E$ and for every $p \in (1,\infty)$ we have the norm estimate
\begin{eqnarray}
  \label{norm_estimate_shift_operator} \| T_m f \|_p \le C  \log\left(2+|m| \right)^\alpha \|f\|_p
\end{eqnarray}
where $\alpha < 1$ depends only on $E$ and $p$, and the constant $C$ depends on $E$, $p$ and $\alpha$ (the same result was formulated for functions $f \colon \R^d \to E$ in 
\cite[Lemma 1]{figiel_singular}). In \cite[Sections 4.3 - 4.5]{mullerpassenbrunner}, Müller and Passenbrunner generalized the definition of shift operators for Christ-type dyadic systems \cite{christ} 
in quasimetric spaces and proved the corresponding $L^p$-estimate for these generalized operators, among other things. Lechner and Passenbrunner then generalized the 
definition further and gave an alternative proof for this norm estimate by providing a way to modify the underlying dyadic system.

In this paper, we revisit and improve some results related to the recent metric adjacent dyadic constructions by Hytönen and the author \cite{hytonentapiola} and give a 
proof for the estimate \eqref{norm_estimate_shift_operator} for UMD-valued functions $f \colon X \to E$ as an application. Our central idea is that with the help of 
adjacent dyadic systems we can split a given dyadic system $\D$ into suitable subcollections $\D_\lambda$ that give us a covenient way to approximate certain 
indicator functions by their conditional expectations. This approximation technique combined with some classical results of UMD-valued analysis give us a fairly 
straightforward proof of the $L^p$ estimate.

\subsection*{Aknowledgements}

This paper is part of the author's PhD project written under the supervision of Professor Tuomas Hytönen.
The author is supported by the European Union through T. Hytönen's ERC Starting Grant 278558 ``Analytic-probabilistic methods for borderline singular integrals'' and 
he is part of Finnish Centre of Excellence in Analysis and Dynamics Research. 

\section{Dyadic cubes, conditional expectations and UMD spaces}

\subsection{Geometrically doubling metric spaces}

Let $(X,d)$ be a geometrically doubling metric space. That is, there exists a constant $M$ such that every ball 
$B(x,r) \coloneqq \{y \in X \colon d(x,y) < r\}$ can be covered by at most $M$ balls of radius $r/2$. In this subsection we do 
not assume any measurability of $(X,d)$ but we note that if $(Y,d',\mu)$ is a doubling metric measure space, then $(Y,d')$ is a 
geometrically doubling metric space.

We use the following two standard lemmas repeatedly in different proofs without referring to them every time we use them.

\begin{lemma}[{\cite[Lemma 2.3]{hytonenframework}}]
  \label{dms_properties}
  The following properties hold for $(X,d)$:
  \begin{enumerate}
   \item[$1)$] Any ball $B(x,r)$ can be covered by at most $\lfloor M\delta^{-\log_2 M} \rfloor$ balls $B(x_i, \delta r)$ for every 
               $\delta \in (0,1]$.
   \item[$2)$] Any ball $B(x,r)$ contains at most $\lfloor M\delta^{-\log_2 M} \rfloor$ centres $x_i$ of pairwise disjoint
               balls $B(x_i,\delta r)$ for every $\delta \in (0,1]$.
  \end{enumerate}
\end{lemma}

\begin{lemma}[{\cite[Lemma 2.2]{hytonentapiola}}]
  \label{maximal_subsets}
  For any $\delta > 0$ there exists a countable \emph{maximal $\delta$-separated set} $\A_\delta \subseteq X$:
     \begin{enumerate}
       \item[$\bullet$] $d(x,y) \ge \delta$ for every $x,y \in \A_\delta$, $x \neq y$
       \item[$\bullet$] $\underset{x \in \A_\delta}{\min} \ d(x,z) < \delta$ for every $z \in X$.
     \end{enumerate}
\end{lemma}

Since the center points of dyadic cybes (see Theorem \ref{thm:dyadic_systems} below) form $\delta^k$-separated sets, the following 
simple lemma is a convenient tool for splitting dyadic systems into smaller sparse systems. We will use the lemma later in Section 
\ref{section:embedding_cubes}.
\begin{lemma}
  \label{lemma:separating_maximal_collections}
  Let $D_2 \ge D_1 > 0$ and let $Z$ be a $D_1$-separated set of points in the space $X$. Then $Z$ is a disjoint union of at most $N$ $D_2$-separated 
  sets where $N$ depends only on $M$ and $D_1 / D_2$.
\end{lemma}

\begin{proof}
  First, notice that any ball of radius $D_2$ can contain at most boundedly many, say $M_1$, 
  points of $Z$ by the second part of Lemma \ref{dms_properties}. By Lemma \ref{maximal_subsets}, we can choose a maximal $D_2$-separated subset $Z_1$ from $Z$.
  By applying the same lemma $M_1$ times, we can choose maximal $D_2$-separated subsets $Z_k \subseteq Z \setminus \bigcup_{i=1}^{k-1} Z_i$ for every 
  $k = 1, 2, \ldots, M_1$. We claim that now $Z \setminus \bigcup_{k=1}^{M_1} Z_k = \emptyset$.
  
  For contradiction, suppose that there exists any point $x \in Z \setminus \bigcup_{k=1}^{M_1} Z_k$. By maximality, $B(x,D_2) \cap Z_k \neq \emptyset$ 
  for every $k = 1, 2, \ldots, M_1$ since otherwise the point $x$ would belong to one of the collections $Z_k$. Thus, the ball $B(x,D_2)$ contains $M_1 + 1$ 
  points of $Z$, which is a contradiction.
\end{proof}

In the construction of metric dyadic cubes we need maximal $\delta^k$ separated sets for every $k \in \Z$. For this we can use Lemma \ref{maximal_subsets} or the 
following stronger result:

\begin{theorem}[{\cite[Theorem 2.4]{hytonentapiola}}]
  For every $\delta \in (0,1/2)$ there exist \emph{maximal nested $\delta^k$-separated sets} $\A_k \coloneqq \{z_\alpha^k \colon \alpha \in \mathcal{N}_k\}$, $k \in \Z$:
  \begin{enumerate}
    \item[$\bullet$] $\A_k \subseteq \A_{k+1}$ for every $k \in \Z$;
    \item[$\bullet$] $d(z_\alpha^k, z_\beta^k) \ge \delta^k$ for $\alpha \neq \beta$;
    \item[$\bullet$] $\min_\alpha d(x,z_\alpha^k) < \delta^k$ for every $x \in X$ and every $k \in \Z$,
  \end{enumerate}
  where $\mathcal{N}_k = \{0,1,\ldots,n_k\}$ if the space $(X,d)$ is bounded, and $\mathcal{N}_k = \N$ otherwise.
\end{theorem}

\subsection{Adjacent dyadic systems in metric spaces}

The following theorem is an improved version of the famous constructions of (quasi)metric dyadic cubes by M. Christ \cite{christ} and E. Sawyer and R. L. 
Wheeden \cite{sawyerwheeden}. This version was proved by Hytönen and A. Kairema \cite[Theorem 2.2]{hytonenkairema} and it has been adapted for 
different dyadic constructions in \cite{hytonentapiola} (see \cite[Theorem 2.9]{hytonentapiola}) and Theorem \ref{thm:adjacent_dyadic_systems} below.

\begin{theorem}
  \label{thm:dyadic_systems}
  Let $(X,d)$ be a doubling metric space and $\delta \in (0,1)$ be small enough. Then for given nested maximal sets of $\delta^k$-separated points 
  $\{z_\alpha^k \colon \alpha \in \A_k\}$, $k \in \Z$, there exist a countable collection of dyadic cubes $\D \coloneqq \{Q_\alpha^k \colon k \in \Z, \alpha \in \A_k\}$ 
  such that
  \begin{enumerate}
    \item[i)] $X = \bigcup_\alpha Q_\alpha^k \ \text{ for every } k \in \Z$;
    \item[ii)] $P,Q \in \D \Rightarrow P \cap Q \in \{\emptyset, P, Q \}$;
    \item[iii)] $B(z_\alpha^k, \frac{1}{5} \delta^k) \subseteq Q_\alpha^k \subseteq B(z_\alpha^k, 3\delta^k)$;
    \item[iv)] $Q_\alpha^k = \bigcup_{\beta: Q_\beta^{k+m} \subseteq Q_\alpha^k} Q_\beta^{k+m}$ for every $m \in \N$.
  \end{enumerate}
\end{theorem}
For every dyadic system $\D$ and cube $Q \coloneqq Q_\alpha^j \in \D$ we use the following notation:
\begin{align*}
  \text{lev}(Q) &\coloneqq j, & &\text{ (level/generation of the cube $Q$)} \\
  \D^k &\coloneqq \{Q_\alpha^k \in \D \colon \alpha \in \A_k\}, & &\text{ (cubes of level $k$)} \\
  B_Q  &\coloneqq B(z_\alpha^j, 3 \delta^j), & &\text{ (ball containing cube $Q$)} \\
  x_Q &\coloneqq z_\alpha^j, & &\text{ (the center point of the cube $Q$)}.
\end{align*}

Like we mentioned earlier, the central idea of our techniques in Section \ref{section:shift_operators} is to split a given dyadic system into suitable 
subcollections that help us approximate certain given indicators by their conditional expectations. For this we use 
\emph{adjacent dyadic systems} which have turned out to be a convenient tool for approximating arbitrary balls and other 
objects by cubes both in $\R^n$ and more abstract settings (see e.g. \cite{kairema, lisun}). In quasimetric spaces they were first constructed by 
Hytönen and Kairema \cite[Theorem 4.1]{hytonenkairema} (based on the ideas of Hytönen and H. 
Martikainen \cite{hytonenmartikainen}) but by restricting ourselves to a strictly metric setting we can use systems with more powerful properties. 
The following theorem was proved recently by Hytönen and the author for $n = 1$:

\begin{theorem}
  \label{thm:adjacent_dyadic_systems}
  Let $(X,d)$ be a doubling metric space with a doubling constant $M$ and let $n \in \N$ be fixed. Then for $\delta < 1 / (n \cdot 168M^8)$ there exist a bounded number of 
  \emph{adjacent} dyadic systems $\D(\omega)$, $\omega = 1,2,\ldots,K = K(\delta)$, such that
  \begin{enumerate}
    \item[I)] each $\D(\omega)$ is a dyadic system in the sense of Theorem \ref{thm:dyadic_systems};
    \item[II)] for a fixed $p \in \N$ and fixed balls $B_1, B_2, \ldots, B_n$ 
               there exist $\omega \in \{1,2,\ldots,K\}$ and cubes $Q_{B_1}, Q_{B_2}, \ldots, Q_{B_n} \in \D(\omega)$ such that for every $i \in \{1,2,\ldots,n\}$ we have
               \begin{enumerate}
                 \item[i)] $B_i \subseteq Q_{B_i}$;
                 \item[ii)] $\ell(Q_{B_i}) \le \delta^{-2} r(B_i)$;
                 \item[iii)] $\delta^{-p} B_i \subseteq Q_{B_i}^{(p)}$,
               \end{enumerate}
               where $\ell(Q) = \delta^k$ if $Q = Q_\alpha^k$, $r(B)$ is the radius of the ball $B$ and $Q_{B_i}^{(p)}$ is the unique dyadic ancestor of $Q_{B_i}$ of generation $\text{\emph{lev}}(Q_{B_i})-p$.
  \end{enumerate}
\end{theorem}

\begin{proof}
  In \cite[Theorem 5.9]{hytonentapiola} the case $n = 1$ was proved by showing that if $B(x,r)$ is a ball such that $\delta^{k+2} < r \le \delta^{k+1}$, then 
  \begin{eqnarray}
    \Pro_\omega \left( \left\{ \omega \in \Omega \colon x \in \left( \bigcup_\alpha \partial_{\delta^{k-p+1}} Q_\alpha^{k-p}(\omega) \cup \bigcup_\alpha \partial_{\delta^{k+1}} Q_\alpha^k(\omega) \right) \right\} \right) \le 168M^8 \delta < 1 \label{proof:adjacent_probability}
  \end{eqnarray}
  where $\Pro_\omega$ is the natural probability measure of the finite set $\Omega \coloneqq \{0,1,\ldots, \lfloor 1 / \delta \rfloor \}$, $Q(\omega)$ is a cube of the 
  dyadic system $\D(\omega)$ and
  \begin{eqnarray*}
    \epsb A \coloneqq \{x \in A \colon d(x,A^c) < \varepsilon \} \cup \{x \in A^c \colon d(x,A) < \varepsilon\}.
  \end{eqnarray*}
  Given \eqref{proof:adjacent_probability}, the proof for general $n \in \N$ is simple. Let $B_1, B_2, \ldots, B_n$ be balls and denote $B_i \coloneqq B(x_i,r_i)$,
  $\delta^{k_i+2} < r_i \le \delta^{k_i+1}$. Then
  \begin{eqnarray*}
   \ \Pro_\omega \left( \left\{ \omega \in \Omega \colon x_i \in \left( \bigcup_\alpha \partial_{\delta^{k_i-p+1}} Q_\alpha^{k_i-p}(\omega) \cup \bigcup_\alpha \partial_{\delta^{k_i+1}} Q_\alpha^{k_i}(\omega) \right) \text{ for some } i \right\} \right) \le n \cdot 168M^8 \delta < 1.
  \end{eqnarray*}
  Thus, there exists $\omega \in \Omega$ such that $x_i \notin \left( \bigcup_\alpha \partial_{\delta^{k_i-p+1}} Q_\alpha^{k_i-p}(\omega) \cup \bigcup_\alpha \partial_{\delta^{k_i+1}} Q_\alpha^{k_i}(\omega) \right)$ for
  every $i = 1,2,\ldots,n$, which is enough to prove the claim.
\end{proof}

\begin{remark}
  \label{remark_adjacent}
  \begin{enumerate}
    \item[1)] In the previous theorem, the constant $K$ is roughly $1/\delta$ \cite[Section 5.2]{hytonentapiola}. Thus, for a large $n$ both the number of 
              systems $\D(\omega)$ and the change of length scale between two consecutive levels of cubes become large.
              
    \item[2)] We will use the previous theorem only for $n = 2$ in the following way. Let $Q_1, Q_2 \in \D^k$ and $m > 1$ be fixed. Then by Theorem \ref{thm:adjacent_dyadic_systems}
              there exists an index $\omega$ and cubes $P_1, P_2 \in \D(\omega)^{k-3}$ such that
              \begin{eqnarray*}
                Q_1 \subseteq B_{Q_1} \subseteq P_1, \ \ \ \ \ Q_2 \subseteq B_{Q_2} \subseteq P_2, \ \ \ \ \ 2mB_{Q_1} \subseteq P_1^{(p_m)}
              \end{eqnarray*}
              for $p_m \in \N$ such that $2m \delta^{p_m} \le 1$.
  \end{enumerate}
\end{remark}

\subsection{Conditional expectations}

Conditional expectations are mostly used in the field of probability theory but they have turned out to be extremely useful also with many questions 
related to more classical analysis (see e.g. \cite{hytonensharpbound}). It is well known among specialists that most of the results related to conditional 
expectations remain true in more general measure spaces but, unfortunately, it is difficult to find a comprehensive presentation of this extended theory in 
the literature. We refer to \cite{tanakaterasawa} for some basic properties of conditional expectations in $\sigma$-finite measure spaces
and \cite[Chapter 9]{williams} for a presentation of the classical probabilistic theory of conditional expectations.

Let $(X,\F,\mu,d)$ be a metric measure space such that $\mu$ is a doubling Borel measure, i.e. there exists a constant $D \coloneqq D_\mu$ such that
\begin{eqnarray*}
  \mu(2B) \le D \mu(B) < \infty
\end{eqnarray*}
for every ball $B$. By construction we know that if $\D$ is a dyadic system given 
by Theorem \ref{thm:dyadic_systems}, then $\D \subseteq \text{Bor} \, X$. In particular, the $\sigma$-algebra generated by 
any subcollection of $\D$ is a subset of $\F$. 

Let us denote $\G_0 \coloneqq \{G \in \G \colon \mu(G) < \infty\}$ for every $\sigma$-algebra $\G \subseteq \F$, and 
let $L^1_\sigma(\G)$ be the space of functions that are integrable over all $G \in \G_0$.

\begin{defin}
  Let $\G$ be $\sigma$-finite sub-$\sigma$-algebra of $\F$ and let $f \colon X \to E$ be a $\F$-measurable function where $E$ is a 
  Banach space. Then a $\G$-measurable function $g$ is a \emph{conditional expectation of $f$ with respect to $\G$} if
  \begin{eqnarray*}
    \int_G f \, d\mu = \int_G \, d\mu
  \end{eqnarray*}
  for every $G \in \G_0$.
\end{defin}
It is not difficult to prove that if the conditional expectation exists, it is unique a.e. Thus, we denote $\E[f | \G] \coloneqq g$ if $g$ is a conditional 
expectation of $f$ with respect to $\G$. Concerning existence, we only need the following elementary case in this paper.

\begin{lemma}
  \label{lemma:conditional_expectation}
  Let $\mathcal{A} \coloneqq \{A_i \colon i \in \N \} \subseteq \F$ be a countable partition of the space $X$ such that $\mu(A_i) < \infty$ for every $i \in \N$ and let $\A$ be the 
  $\sigma$-algebra generated by $\mathcal{A}$. Then for every $f \in L^1_\sigma(\F)$ we have
  \begin{eqnarray*}
    \E[f | \A] = \sum_{A \in \mathcal{A}} 1_A \langle f \rangle_A.
  \end{eqnarray*}
\end{lemma}

\begin{proof}
  Let $G \in \A_0$. Then there exist pairwise disjoint sets $A_1^G, A_2^G, \ldots \in \mathcal{A}$ such that $G = \bigcup_i A_i^G$.
  Now
  \begin{eqnarray*}
    \int_G f \, d\mu = \sum_i \int_{A_i^G} \left( \fint_{A_i^G} f \, d\mu \right) d\mu = \int_G \sum_i 1_{A_i^G} \left( \fint_{A_i^G} f \, d\mu \right) d\mu = \int_G \left( \sum_{A \in \mathcal{A}} 1_A \fint_A f \, d\mu \right) d\mu
  \end{eqnarray*}
  which proves the claim.
\end{proof}

\subsection{UMD spaces; type and cotype of Banach spaces}

Let $(X,d,\F,\mu)$ be a metric measure space and let $(\F_k)$, $k = 0,1,\ldots,N$, be a sequence of sub-$\sigma$-algebras of $\F$ such that $\F_k \subseteq \F_{k+1}$ for all $k$. For simplicity, let us denote
\begin{eqnarray*}
  \|\cdot\|_p &\coloneqq& \|\cdot\|_{L^p(X;E)}
\end{eqnarray*}
where $\| \cdot \|_{L^p(X;E)}$ is the $L^p$-Bochner norm.

\begin{defin}
  A sequence of functions $(d_k)_{k=1}^N$ is a \emph{martingale difference sequence} if $d_k$ is $\F_k$-measurable and $\E[d_k | \F_{k-1}] = 0$ for every $k$.
\end{defin}

\begin{defin}
  A Banach space $(E,\|\cdot\|_E)$ is a \emph{UMD} (\emph{unconditional martingale difference}) \emph{space} if for every $p \in (1,\infty)$ there exists a constant $\beta_p$ such that
  \begin{eqnarray*}
    \left\| \sum_{i=1}^N \varepsilon_i d_i \right\|_p \le \beta_p \left\| \sum_{i=1}^N d_i \right\|_p
  \end{eqnarray*}
  for all $E$-valued $L^p$-martingale difference sequences $(d_i)_{i=1}^N$ (i.e. $(d_i)$ is a martingale difference sequence such that $d_i \in L^p(X,\F_i;E)$ for every $i$) 
  and for all choices of signs $(\varepsilon_i)_{i=1}^N \in \{-1,+1\}^N$.
\end{defin}
UMD spaces are crucial in Banach space valued harmonic analysis due to their many good properties; for example, 
a Banach space $E$ is a UMD space if and only if the Hilbert transform is bounded on $L^p(\R;E)$ \cite{burkholder, bourgain}. They give us a 
natural setting for analysis that is based on techniques used in probability spaces in the following way.
Let $(d_i)$ be a martingale difference sequence and let $(\varepsilon_i)$ be a sequence of \emph{random signs}, i.e. independent random variables on some probability space $(\Omega,\Pro)$, 
with distribution $\Pro\left( \varepsilon_i = -1 \right) = \Pro\left( \varepsilon_i = +1 \right) = 1/2$. Then for every $\eta \in \Omega$ the sequence 
$(\varepsilon_i(\eta) d_i)$ is a martingale difference sequence. In particular, the UMD property gives us
\begin{eqnarray}
  \label{inequality_random_signs} \left\| \sum_{i=1}^N d_i \right\|_p \eqsim_E \left( \int_\Omega \left\| \sum_{i=1}^N \varepsilon_i(\eta) d_i \right\|^p_p \, d\Pro(\eta) \right)^{1/p} \eqqcolon \left\| \sum_{i=1}^N \varepsilon_i d_i \right\|_{\Omega,p}.
\end{eqnarray}
for every $p \in (1,\infty)$.

The following inequality by J. Bourgain is a standard tool in UMD valued analysis. Its original scalar-valued version was due to E. Stein.
\begin{theorem}[See e.g. {\cite[Proposition 3.8]{clementetall}}]
  \label{inequality:stein}
  Let $(f_k)$ be a sequence of functions in $L^p(X,\F;E)$ and $(\F_k)$ a sequence of $\sigma$-finite $\sigma$-algebras such that $\F_k \subseteq \F_{k+1} \subseteq \F$
  for every $k \in \N$. Then for any sequence of random signs $(\varepsilon_k)$ we have
  \begin{eqnarray*}
    \left\| \sum_k \varepsilon_k \E[f_k | \F_k] \right\|_{\Omega,p} \lesssim_{p,\beta_p} \left\| \sum_k \varepsilon_k f_k \right\|_{\Omega,p}.
  \end{eqnarray*}
\end{theorem}

In our proofs we also need the following version of the well-known principle of contraction by J.-P. Kahane. It holds in all Banach spaces.
\begin{theorem}[ {\cite[Theorem 5 (Section 2.6)]{kahane} }]
  \label{kahane_contraction_principle}
  Suppose that $(\varepsilon_i)$ is a sequence of random signs and the series $\sum_i \varepsilon_i x_i$ converges in $E$ almost surely.
  Then for any bounded sequence of scalars $(c_i)$ the series $\sum_i \varepsilon_i c_i x_i$ converges in $E$ almost surely and
  \begin{align*}
    \int_\Omega \left\| \sum_i \varepsilon_i c_i x_i \right\|_E^p d\Pro \le \left(\sup_i |c_i| \right)^p \int_\Omega \left\| \sum_i \varepsilon_i x_i \right\|_E^p d\Pro.
  \end{align*}
\end{theorem}

\subsubsection{Type and cotype of Banach spaces}
\label{subsubsection:type}

\begin{defin}
  Let $(E,\|\cdot\|)$ be a Banach space. We say that $E$ has \emph{type} $t \in [1,2]$ if there exists a constant $C_t > 0$ such that for every 
  finite sequence $(x_i)$ in $E$ and finite sequence $(\varepsilon_i)$ of random signs we have
  \begin{eqnarray*}
    \int_\Omega \left\| \sum_i \varepsilon_i x_i \right\|_E d\Pro \le C_t \left( \sum_i \|x_i\|^t \right)^{1/t}.
  \end{eqnarray*}
  In a similar fashion, we say that $E$ has \emph{cotype} $q \in [2,\infty]$ if there exists a constant $C_q > 0$ such that
  \begin{eqnarray*}
    \left( \sum_i \|x_i\|^q \right)^{1/q} \le C_q \int_\Omega \left\| \sum_i \varepsilon_i x_i \right\|_E d\Pro.
  \end{eqnarray*}
\end{defin}
The notion of type and cotype of Banach spaces was introduced by B. Maurey and G. Pisier in the 1970's and it has become an important part 
of analysis on Banach spaces. Out of this rich theory, we need the following results:
\begin{enumerate}
  \item[i)] If $Y$ is a $\sigma$-finite measure space and $E$ is a Banach space of type $r$ and cotype $s$, then $L^p(X;E)$ has type $\min\{p,r\}$ and cotype $\max\{p,s\}$.
  \item[ii)] If $E$ is a UMD space, then $E$ has a non-trivial type $s > 1$ and non-trivial cotype $t < \infty$.
\end{enumerate}
For proofs, see e.g. \cite[Chapter 9]{ledouxtalagrand} for i) and \cite[Theorem 11.1.14]{albiackalton}, \cite[Proposition 3]{rubiodefrancia} for ii).

\subsection{Structural constants}
 
We say that $c$ is a \emph{structural constant} if it depends only on the doubling constant $D$, the UMD constant $\beta_p$ for a fixed $p \in (1,\infty)$ and 
the type and cotype constants $C_t$ and $C_q$. We do not track the dependencies of our bounds on the structural constants and thus, we use the notation $a \lesssim b$ if 
$a \le cb$ for some structural constant $c$ and $a \eqsim b$ if $a \lesssim b \lesssim a$.

\section{Embedding cubes into larger cubes}
\label{section:embedding_cubes}

In this section we prove a decomposition result for dyadic systems using Theorem \ref{thm:adjacent_dyadic_systems}.
We formulate the result in such a way that it is easy to apply it in Section \ref{section:shift_operators} but we note
that it is simple to modify the proof for other similar decompositions.

Let $\D$ be a dyadic system with $\delta < 1 / (2 \cdot 168M^8)$ and $\{\D_\omega\}_\omega$ be adjacent dyadic systems for 
the same $\delta$ given by Theorem \ref{thm:adjacent_dyadic_systems}. Let us fix a number $m \ge 1$ and an injective 
function $\tau \colon \D \to \D$ such that $\tau(Q) \subseteq mB_Q$ for every $Q = Q_\alpha^k \in \D$ and $\tau \D^k \subseteq \D^k$ for 
every $k \in \Z$.

\begin{proposition}
  \label{proposition_dyadic_decomposition}
  The system $\D$ is a disjoint union of a bounded number of subcollections $\D_\lambda \subseteq \D$, $\lambda = (i,j,\omega)$, 
  with the following property: for every $Q \in \D_\lambda$ there exist cubes $P_Q, P_{\tau(Q)} \in \D(\omega)^{k-3}$ and $P_Q^* \in \D(\omega)^{k-3-T}$, where
  $2m\delta^T \le 1$, such that
\begin{eqnarray}
  \label{partition_property1} && Q \subseteq P_Q, \ \ \ \ \tau(Q) \subseteq P_{\tau(Q)}, \ \ \ \ P_Q \cup P_{\tau(Q)} \cup 2mB_Q \subseteq P_Q^*; \\
  \label{partition_property2} && \text{ if } Q_1, Q_2 \in \D_\lambda \cap \D^k, Q_1 \neq Q_2, \text{ then } (P_{Q_1} \cup P_{\tau(Q_1)}) \cap (P_{Q_2} \cup P_{\tau(Q_2)}) = \emptyset; \\
  \label{partition_property3} && \text{ if } Q_1, Q_2 \in \D_\lambda, Q_1 \subsetneq Q_2, \text{ then } P_{Q_1}^* \subseteq P_{Q_2}.
\end{eqnarray}
\end{proposition}
In other words, we split the collection $\D$ into sparse subcollections $\D_\lambda$ such that we can embed 
every cube $Q \in \D_\lambda$ and its image $\tau(Q)$ into some larger cubes $P_Q$ and $P_{\tau(Q)}$ such that $P_Q$ and $P_{\tau(Q)}$ belong
to the same dyadic system and they have a mutual dyadic ancestor $P_Q^*$. 

We form the sets $\D_\lambda$ with the help of next technical lemma.
\begin{lemma}
  \label{lemma:partition_of_dyadic_system}
  The collection $\D$ is a disjoint union of $L = L(X)$ subcollections $\Q_i$ such that for every $k \in \Z$ and $Q_1,Q_2 \in \Q_i \cap \D^k$ we have
  \begin{eqnarray*}
    3\delta^{-3}B_{R_1} \cap 3\delta^{-3} B_{R_2} = \emptyset
  \end{eqnarray*}
  where $R_1 \in \{Q_1, \tau(Q_1)\}$ and $R_2 \in \{Q_2, \tau(Q_2)\}$, $R_1 \neq R_2$, and the number $L$ is independent of $m$.
\end{lemma}

\begin{proof}
  Basically, we only need to use basic properties of geometrically doubling metric spaces with the help of the observation that if $Q, P \in \D^k$ and $d(x(Q),x(P)) \ge 12\delta^{k-3}$,
  then $3\delta^{-3}B_Q \cap 3\delta^{-3}B_P = \emptyset$.
  
  Let $k \in \Z$ be fixed. For any subcollection $\Q \subseteq \D^k$ and any set $A$ of center points of cubes, let us denote 
  \begin{eqnarray*}
    Y_\Q &\coloneqq& \{ x(Q) \colon Q \in \Q\}, \\
    \D_A &\coloneqq& \{ Q \in \D \colon x(Q) \in A\}.
  \end{eqnarray*}
  We split the set $Y_{\D^k}$ into smaller sets in three steps. To keep our notation simple, $i$ is an index whose role may change from one occurence to the next.
  \begin{enumerate}
    \item[1)] By Lemma \ref{lemma:separating_maximal_collections}, we can split the $\delta^k$-separated set $Y_{\D^k}$ into a bounded number of $12\delta^{k-3}$-separated subsets $Y_{i,k}^1$.
    \item[2)] For every $Q \in \D_{Y_{i,k}^1}$, the ball $3\delta^{-3} B_Q$ intersects at most a bounded number of balls $3\delta^{-3}B_{\tau(P)}$ where $P \in \D_{Y_{i,k}^1}$. Thus, we can split the set
              $Y_{i,k}^1$ into a bounded number of subsets $Y_{i,k}^2$ such that $3\delta^{-3} B_Q \cap 3\delta^{-3} B_{\tau(P)} = \emptyset$ for every $Q,P \in \D_{Y_{i,k}^2}$, $Q \neq \tau(P)$.
    \item[3)] For every $Q \in \D_{Y_{i,k}^2}$, the ball $3\delta^{-3} B_{\tau(Q)}$ intersects at most a bounded number of balls $3\delta^{-3} B_{\tau(P)}$, $P \in \D_{Y_{i,k}^2}$. Thus, we can split the set 
              $Y_{i,k}^2$ into a bounded number of subsets $Y_{i,k}^3$ such that $3\delta^{-3} B_{\tau(Q)} \cap 3\delta^{-3} B_{\tau(P)} = \emptyset$ for every $Q,P \in \D_{Y_{i,k}^3}$, $Q \neq P$.
  \end{enumerate}
  Now we can set $\Q_i \coloneqq \bigcup_{k \in \Z} \D_{Y_{i,k}^3}$ for every $i$.
\end{proof}

Let $\{\Q_i\}_i$ be the partition of $\D$ given by the previous lemma and let $T \in \N$, $T \ge 1$, be the smallest number such that $$ 2m \delta^T \le 1.$$
Recall Theorem \ref{thm:adjacent_dyadic_systems} and denote
\begin{eqnarray*}
  \gamma(R) \coloneqq \min\left\{ \omega \colon Q_{B_R}, Q_{B_{\tau(R)}} \in \D(\omega), \delta^{-T} B_R \subseteq Q_{B_R}^{(T)}\right\}
\end{eqnarray*}
for every cube $R \in \D$ and
\begin{eqnarray*}
  \Q_{i,\omega} \coloneqq \{ R \in \Q_i \colon \gamma(R) = \omega\}
\end{eqnarray*}
for every $i = 1,2,\ldots,L$ and $\omega = 1,2,\ldots,K$. Then the collections $\Q_{i,\omega}$ satisfy properties \eqref{partition_property1} and \eqref{partition_property2}
but they are still not suitable for property \eqref{partition_property3}. Thus, we split collections $Q_{i,\omega}$ into smaller collections whose cubes 
have large enough generation gaps: we set
\begin{eqnarray*}
  \D_{i,j,\omega} \coloneqq \bigcup_{k \in \Z} \left( \Q_{i,\omega} \cap \D^{j + 4kT} \right)
\end{eqnarray*}
for every $j = 0,1,\ldots,4T-1$. Notice that the indices $i$, $j$ and $\omega$ are independent of each other.

\begin{proof}[Proof of Proposition \ref{proposition_dyadic_decomposition}]
  Clearly we only need to show the claim for the collections $\D_{i,0,\omega} \eqqcolon \D_i$. Recall 

  Notice first that 
    $$2m \cdot r(B_Q) = 6m \delta^{4kT} \le \delta^{-T} 3\delta^{4kT} = \delta^{-T} \cdot r(B_Q)$$ 
  for every $Q \coloneqq Q_\alpha^{4kT} \in \D_i$. Thus, by Remark \ref{remark_adjacent}
  and the definition of $\D_i$, for every cube $Q \in \D_i$ there exist cubes $P_Q, P_{\tau(Q)} \in \D(\omega)^{4kT - 3}$ such that 
  \begin{eqnarray*}
    B_Q \subseteq P_Q, \ \ \ \ B_{\tau(Q)} \subseteq P_{\tau(Q)}, \ \ \ \ 2mB_Q \subseteq P_Q^{(T)} \eqqcolon P_Q^*.
  \end{eqnarray*}
  Let us then show that the cubes $P_Q$, $P_{\tau(Q)}$ and $P_Q^*$ satisfy properties \eqref{partition_property1} - \eqref{partition_property3}.
  \begin{enumerate}
    \item[\underline{\eqref{partition_property1}} \ ] Since $Q, \tau(Q) \subseteq 2mB_Q$, we know that $P_Q \cap P_Q^* \neq \emptyset$ and $P_{\tau(Q)} \cap P_Q^* \neq \emptyset$. 
                                                      Thus, since $\D(\omega)$ is a dyadic system and $\text{lev}(P_Q^*) < \text{lev}(P_Q) = \text{lev}(P_{\tau(Q)})$, 
                                                      we have $P_Q \cup P_{\tau(Q)} \subseteq P_Q^*$.
                                                      
                                                      \
  
    \item[\underline{\eqref{partition_property2}} \ ] Since $x(Q) \in P_Q$ for every cube $Q \in \D$, we have 
                                                        $$P_Q \subseteq B(x(P_Q),3\delta^{4kT-3}) \subseteq B(x(Q),6\delta^{4kT-3}) = 2\delta^{-3} B_Q$$
                                                      for every cube $Q \in \D$. Thus, the property \eqref{partition_property2} follows directly 
                                                      from Lemma \ref{lemma:partition_of_dyadic_system}.
                                                      
                                                      \
                                                      
    \item[\underline{\eqref{partition_property3}} \ ] Suppose that $R \subsetneq Q \coloneqq Q_\alpha^{4kT}$. Then $\text{lev(R)} \ge (4k+4)T$ and thus, 
                                                      $\text{lev}(P_R) \ge (4k+4)T - 3$ and
                                                       $$\text{lev}(P_R^*) \ge (4k+4)T - 3 - T \ge 4kT = \text{lev}(Q) \ge \text{lev}(P_Q)$$ 
                                                      since $T \ge 1$. In particular, $P_R^* \subseteq P_Q$ since $P_R^*, P_Q \in \D(\omega)$ and $\D(\omega)$ is a dyadic system.
  \end{enumerate}
\end{proof}
 
\section{$L^p$-boundedness of shift operators}
\label{section:shift_operators}

In this section, we show that with the help of Proposition \ref{proposition_dyadic_decomposition} we can give a straightforward proof for the $L^p$-boundedness of the shift operators in doubling 
metric measure spaces. We follow some ideas of \cite{figiel_haar} and \cite{lechnerpassenbrunner} but mostly we rely on our own dyadic constructions.

Let $(X,d)$ be a metric space, $\mu$ a doubling Borel measure on $X$ and $(E,\|\cdot\|)$ an UMD space. Since the doubling property of $\mu$ implies the geometrical 
doubling property of $d$, there exists a finite geometrical doubling constant $M$. Thus, we may fix a dyadic system $\D$ for $\delta < 1 / (2 \cdot 168M^8)$ and
adjacent dyadic systems $\{\D(\omega)\}_\omega$ given by Theorem \ref{thm:dyadic_systems} for the same $\delta$.

\subsection{Haar functions}
There are various different ways to construct Haar functions in metric spaces (see e.g. \cite[Section 5]{aimaretall}) 
and thus, we do not want to fix any particular construction. We do, however, refer to the construction in \cite[Section 4]{hytonennonhomogeneous} 
(with the choice $b \equiv 1$) for a system of Haar functions that satisfy the properties in the following definition. In \cite{hytonennonhomogeneous} the 
construction is done in $\R^n$ for a non-doubling measure but it is simple to generalize the result for our setting.

\begin{defin}
  A collection of functions $h_Q^\theta \colon X \to \R$, $Q \coloneqq Q_\alpha^k \in \D$, $\theta = 1,\ldots,n(Q) \le \Theta$, is a \emph{system of Haar functions} if it satisfies the following properties:
  for every $Q$ and $\theta$ we have
  \begin{enumerate}
    \item[$\bullet$] $\text{supp} \, h_Q^\theta \subseteq Q$;
    \item[$\bullet$] $h_Q^\theta$ is constant on every child cube $Q_\beta^{k+1} \subseteq Q$;
    \item[$\bullet$] $\int h_Q^\theta = 0 = \int h_Q^\theta h_Q^{\theta'}$ \ if \ $\theta \neq \theta'$;
    \item[$\bullet$] $\| h_Q^\theta \|_2 = 1$;
  \end{enumerate}
  and the space of finite linear combinations of the functions $h_Q^\theta$ is dense in $L^2(X;E)$. 
\end{defin}
The number $\Theta$ in the previous definition depends only on $M$ or, more precisely, the maximum number of child cubes 
$Q_\beta^{k+1}$ a cube $Q_\alpha^k$ can have. Henceforth, we fix some $\theta = \theta(Q)$ for each $Q \in \D$ and drop 
the dependency on $\theta$ in the notation.

Let $h_Q = \sum_k v_k 1_{Q_k}$ be a Haar function, where $Q_k$ are the child cubes of $Q$. The following properties are straightforward consequences 
of the previous definition:
\begin{eqnarray}
  \| h_Q \|_\infty &=& \max |v_k| \ \ \eqsim \ \ \frac{1}{\mu(Q)^{1/2}};\\
  \| h_Q \|_1 &\eqsim& \mu(Q)^{1/2}.
\end{eqnarray}
In particular, 
\begin{eqnarray}
  \label{absolute_value_of_haar_functions} \frac{1_{Q_k}(x)}{\mu(Q_k)^{1/2}} \lesssim |h_Q(x)| \lesssim \frac{1_Q(x)}{\mu(Q)^{1/2}} \ \ \ \ \ \text{ for every } x \in Q \text{ and some } Q_k.
\end{eqnarray}
The previous properties give us the following lemma:
\begin{lemma}
  \label{lemma:reduction_from_haar_to_indicators}
  For every $p \in (1,\infty)$ and finite collection of cubes $Q$ we have
  \begin{eqnarray*}
    \left\| \sum_Q x_Q h_Q \right\|_p \eqsim \left\| \sum_Q \varepsilon_Q x_Q \frac{1_Q}{\mu(Q)^{1/2}} \right\|_{\Omega,p}.
  \end{eqnarray*}
\end{lemma}

\begin{proof}
  Let us denote $\sum_Q x_Q h_Q = \sum_k \sum_{\alpha} x_{Q_{\alpha}^k} h_{Q_{\alpha}^k}$ and let $(\varepsilon_Q)$ be a sequence of random signs. Then for every $y \in X$ and $k \in \Z$ there exists 
  at most one ${Q_{\alpha,y}^k}$ such that $h_{Q_{\alpha,y}^k}(y) \neq 0$. Let $\sigma_k^y \in \{-1,+1\}$ be such that $\sigma_k^y h_{Q_{\alpha,y}^k}(y) = |h_{Q_{\alpha,y}^k}(y)|$ for every $y \in X$ and $k \in \Z$.
  Then, for a fixed $y \in X$, $(\sigma_k^y \varepsilon_{Q_{\alpha,y}^k})_k$ is a sequence of random signs. Since the functions $h_Q$ form a martingale difference sequence and 
  by \eqref{absolute_value_of_haar_functions} we know that $|h_Q|\mu(Q)^{1/2} \lesssim 1$ for every $Q$, we have
  \begin{eqnarray*}
    \left\| \sum_Q x_Q h_Q \right\|_p^p &\eqsim& \int_X \int_\Omega \left\| \sum_k \sigma_k^y \varepsilon_{Q_{\alpha,y}^k}(\eta) x_{Q_{\alpha,y}^k} h_{Q_{\alpha,y}^k}(y) \right\|_E^p \, d\Pro(\eta) \, d\mu(y) \\
                                        &=& \int_X \int_\Omega \left\| \sum_k \varepsilon_{Q_{\alpha,y}^k}(\eta) \frac{x_{Q_{\alpha,y}^k}}{\mu(Q_{\alpha,y}^k)^{1/2}} \left| h_{Q_{\alpha,y}^k}(y) \right| \mu({Q_{\alpha,y}^k})^{1/2} \right\|_E^p \, d\Pro(\eta) \, d\mu(y) \\
                                        &\lesssim& \left\| \sum_Q \varepsilon_Q x_Q \frac{1_Q}{\mu(Q)^{1/2}} \right\|_{\Omega,p}^p. 
  \end{eqnarray*}
  by the UMD property of $E$, Fubini's theorem and Kahane's contraction principle. Let us then denote $\sum_Q x_Q h_Q = \sum_{i=1}^N x_i h_{Q_i}$ where $\text{lev}(Q_1) \le \text{lev}(Q_2) \le \ldots \le \text{lev}(Q_N)$. Then by Lemma \ref{lemma:conditional_expectation} we have $\E[ |h_{Q_i}| | \F_i] = 1_Q \langle |h_Q| \rangle_Q$
  where $\F_i$ be the $\sigma$-algebra generated by $\D^{\text{lev}(Q_i)}$. Thus, since $1 / (\mu(Q)^{1/2} \langle |h_Q| \rangle_Q ) \eqsim 1$, the previous estimates, 
  Stein's inequality and Kahane's contraction principle (in this order) give us
  \begin{eqnarray*}
    \left\| \sum_{i=1}^N x_i h_{Q_i} \right\|_p^p \ \ \eqsim \ \ \left\| \sum_{i=1}^N \varepsilon_i x_i |h_{Q_i}| \right\|_{\Omega,p}^p &\gtrsim& \left\| \sum_{i=1}^N \varepsilon_i x_i \E[|h_{Q_i}| | \F_i] \right\|_{\Omega,p}^p \\
    &=& \left\| \sum_i^N \varepsilon_i x_i \frac{1_{Q_i}}{\mu(Q_i)^{1/2}} \mu(Q_i)^{1/2} \langle |h_{Q_i}| \rangle_{Q_i} \right\|_{\Omega,p}^p \\
    &\gtrsim& \left\| \sum_i^N \varepsilon_i x_i \frac{1_{Q_i}}{\mu(Q_i)^{1/2}}\right\|_{\Omega,p}^p,
  \end{eqnarray*}
  which proves the claim.
\end{proof}

\subsection{Shift operators}
Let us fix the number $m \ge 1$ and let $\tau \colon \D \to \D$ be an injective function such that
\begin{enumerate}
  \item[1)] $\tau \D^k \subseteq \D^k$ for every $k \in \Z$;
  \item[2)] for every $Q \in \D$ we have $\tau(Q) \subseteq mB_Q$;
  \item[3)] the measures of cubes $Q$ and $\tau(Q)$ are approximately the same:
            \begin{eqnarray}
              \label{tau_condition} \mu(Q) \eqsim \mu(\tau(Q)).
            \end{eqnarray}
\end{enumerate}
Let $\{h_Q\}_{Q \in \D}$ be a system of Haar functions. Then we can define the shift operator $T \coloneqq T_\tau$ as 
the linear extension of the operator $\hat{T}$,
\begin{eqnarray*}
  \hat{T} h_Q = h_{\tau(Q)}.
\end{eqnarray*}
It is easy to see that without condition \eqref{tau_condition} an estimate of the type \eqref{norm_estimate_shift_operator} is out of reach for all $p \in (1,\infty)$. 
More precisely: by property \eqref{absolute_value_of_haar_functions} we have $\| h_Q \|_p \eqsim \mu(Q)^{1/p - 1/2}$ for every cube $Q$ and thus, without condition 
\eqref{tau_condition} the estimate cannot hold simultaneously for all $p \in (1,2]$ and for all $q \in (2,\infty)$.
We note that the condition \eqref{tau_condition} is automatically valid in metric measure spaces that satisfy an Ahlfors-regularity type condition.

\subsection{$L^p$-boundedness of shift operators}

Using Proposition \ref{proposition_dyadic_decomposition} and Lemma \ref{lemma:reduction_from_haar_to_indicators} we can now prove the following theorem quite easily.

\begin{theorem}
  \label{thm:Lp-boundedness_of_shift_operators}
  Let $p \in (1,\infty)$ and $f \in L^p(X;E)$. Then
  \begin{eqnarray*}
    \| Tf \|_p \le C \left( \log (2m) + 1 \right)^\alpha \| f \|_p
  \end{eqnarray*}
  where $C = C(p,X,E,\alpha)$, $\alpha = 1 / \min\{t_E,p\} - 1 / \max\{q_E,p\} < 1$ and $t_E$ and $q_E$ are the type and cotype of the space $E$.
\end{theorem}

\begin{proof}
  Suppose that $f \in L^p(X;E)$. Then, by the properties of the Haar functions and Proposition \ref{proposition_dyadic_decomposition}, we may assume that the 
  function $f$ is of the form
  \begin{eqnarray*}
    f = \sum_{i=1}^L \sum_{j=0}^{4T-1} \sum_{\omega=1}^K \sum_{Q \in \D_{i,j,\omega}} x_Q h_Q
  \end{eqnarray*}
  where $x_Q \neq 0$ only for finitely many $Q$. Thus, we can denote $f = \sum_{i,j,\omega} \sum_{k=1}^n x_k h_{Q_k}$ where 
  $\lev(Q_1) \le \lev(Q_2) \le \ldots \le \lev(Q_n)$.
    
  For every $k = 1,2,\ldots,n$, let $\F_k$ be the $\sigma$-algebra generated by
  \begin{eqnarray*}
    F_k \coloneqq \left(\D(\omega)^{\lev(Q_k)-3} \setminus \underset{\lev(Q_l) = \lev(Q_k)}{\bigcup_{l=1,\ldots,n}} \left\{P_{Q_l}, P_{\tau(Q_l)}\right\} \right) \cup \underset{\lev(Q_l) = \lev(Q_k)}{\bigcup_{l=1,\ldots,n}} \left\{P_{Q_l} \cup P_{\tau(Q_l)}\right\}.
  \end{eqnarray*}
  Notice that if $\lev(Q_{k_1}) = \lev(Q_{k_2})$, then $F_{k_1} = F_{k_2}$. By property \eqref{partition_property2} we know that $F_k$ is a partition of the space $X$ and by property \eqref{partition_property3} we know that the sequence $(\F_k)$ is nested.
  Thus, for every $k = 1,2,\ldots,n$ we have
  \begin{eqnarray}
    \E[ 1_{Q_k} | \F_k] \, \overset{\ref{lemma:conditional_expectation}}{=} \, 1_{P_{Q_k} \cup P_{\tau(Q_k)}} \langle 1_{Q_k} \rangle_{P_{Q_k} \cup P_{\tau(Q_k)}} 
    \overset{\eqref{tau_condition}}{\eqsim} 1_{P_{Q_k} \cup P_{\tau(Q_k)}} \frac{\mu(Q_k)}{\mu(P_{Q_k})} \, \eqsim \, 1_{P_{Q_k} \cup P_{\tau(Q_k)}}. && \label{main_theorem_conditional_expectation}
  \end{eqnarray}
  In particular, 
  \begin{eqnarray*}
    \left\| \sum_k x_k h_{\tau(Q_k)} \right\|_p 
    \ \overset{\ref{lemma:reduction_from_haar_to_indicators}}{\eqsim} \ \left\| \sum_k \varepsilon_k \frac{x_k}{\mu(\tau(Q_k))^{1/2}} 1_{\tau(Q_k)} \right\|_{\Omega,p}
    &\overset{\ref{kahane_contraction_principle}}{\overset{\eqref{tau_condition}}{\lesssim}}& \left\| \sum_k \varepsilon_k \frac{x_k}{\mu(Q_k)^{1/2}} 1_{P_{Q_k} \cup P_{\tau(Q_k)}} \right\|_{\Omega,p} \\
    &\overset{\eqref{main_theorem_conditional_expectation}}{\eqsim}& \left\| \sum_k \varepsilon_k \frac{x_k}{\mu(Q_k)^{1/2}} \E[ 1_{Q_k} | \F_k] \right\|_{\Omega,p} \\
    &\overset{\ref{inequality:stein}}{\lesssim}& \left\| \sum_k \varepsilon_k \frac{x_{Q_k}}{\mu(Q_k)^{1/2}} 1_{Q_k} \right\|_{\Omega,p}.
  \end{eqnarray*}
  Hence, since by Section \ref{subsubsection:type} the space $L^p(X;E)$ has a non-trivial type $t > 1$ and a non-trivial cotype $q < \infty$, we have
  \begin{eqnarray*}
    \left\| \sum_{i,j,\omega} \sum_k x_k h_{\tau(Q_k)} \right\|_p  &\lesssim& \left( \sum_{i,j,\omega} \left\| \sum_k \varepsilon_k \frac{x_{Q_k}}{\mu(Q_k)^{1/2}} 1_{Q_k} \right\|_{\Omega,p}^t \right)^{1/t} \\
                                                                   &\le& \left( 4TKL \right)^{1/t - 1/q} \left( \sum_{\omega,i,j} \left\| \sum_k \varepsilon_k \frac{x_{Q_k}}{\mu(Q_k)^{1/2}} 1_{Q_k} \right\|_{\Omega,p}^q \right)^{1/q} \\
                                                                   &\lesssim& T^{1/t - 1/q} \left\| \sum_{\omega,i,j} \sum_k \varepsilon_k \frac{x_{Q_k}}{\mu(Q_k)^{1/2}} 1_{Q_k} \right\|_{\Omega,p} \\
                                                                   &\overset{\ref{lemma:reduction_from_haar_to_indicators}}{\lesssim}& (\log (2m) + 1)^{1/t - 1/q} \| f \|_p
  \end{eqnarray*}
  by Hölder's inequality.
\end{proof}

\bibliography{embedded}
\bibliographystyle{plain}
\end{document}